\documentclass[10pt, letterpaper, conference]{ieeeconf}
\IEEEoverridecommandlockouts
\usepackage{amsmath, amsfonts, amssymb, mathtools, url, graphicx, graphicx}
\usepackage{newtxtext}
\usepackage{newtxmath}
\let\labelindent\relax
\usepackage{enumitem}
\usepackage[usenames, dvipsnames]{color}
\usepackage{algorithm,algorithmic}

\allowdisplaybreaks
\usepackage{tikz}
\usetikzlibrary{automata,positioning}
\usepackage{caption}
\usepackage{subcaption}
\usepackage{algorithm,algorithmic}

\newtheorem{proposition}{Proposition}
\newtheorem{defn}{Definition}
\newtheorem{assump}{Assumption}

\newtheorem{prob}{Problem}
\newtheorem{remark}{Remark}

\newcommand{\norm}[1]{\left\lVert{#1}\right\rVert}
\newcommand{\abs}[1]{\left\lvert{#1}\right\rvert}

\renewcommand{\geq}{\geqslant}

\renewcommand{\leq}{\leqslant}

\newcommand{\R}{\mathbb{R}}
\newcommand{\N}{\mathbb{N}}
\renewcommand{\P}{\mathcal{P}}

\newcommand{\Sw}{\mathcal{S}}

\newcommand{\K}{\mathcal{K}}
\newcommand{\KL}{\mathcal{KL}}
\newcommand{\Kinfty}{\mathcal{K}_{\infty}}

\newcommand{\EE}{\mathrm{E}}

\allowdisplaybreaks

\title{A randomized algorithm for the stabilization of switched nonlinear systems\\under restricted switching}
\author{Atreyee Kundu%
	\thanks{The author is with the Department of Electrical Engineering, Indian Institute of Science Bangalore, Bengaluru - 560012, Karnataka, India, E-mail: \texttt{atreyeek@iisc.ac.in}}
}

\begin{document}

	\maketitle

	\begin{abstract}
        This paper deals with input/output-to-state stability (IOSS) of switched nonlinear systems in the discrete-time setting. We present an algorithm to construct periodic switching signals that obey pre-specified restrictions on admissible switches between the subsystems and admissible dwell times on the subsystems, and identify sufficient conditions on the individual subsystems, the admissible switches and admissible dwell times under which a switching signal obtained from our algorithm preserves stability of a switched system with overwhelming probability. We recover our earlier result on probabilistic techniques for the design of switching signals that preserve global asymptotic stability of switched linear systems under sufficient conditions on the properties of the individual subsystems and the admissible dwell times on the subsystems.
    \end{abstract}
\section{Introduction}
\label{s:intro}
\subsection{The problem}
\label{ss:prob_stat}
    A \emph{switched system} has two ingredients --- a family of systems and a switching signal. The \emph{switching signal} selects an \emph{active subsystem} at every instant of time, i.e., the system from the family that is currently being followed \cite[\S 1.1.2]{Liberzon}. Switched systems find wide applications in power systems and power electronics, aircraft and air traffic control, network and congestion control, etc. \cite[p.\ 5]{Sun}. 

    We consider a discrete-time switched system with inputs and outputs
    \begin{align}
    \label{e:swsys}
        x(t+1) &= f_{\sigma(t)}(x(t),v(t)),\:\:x(0) = x_{0},\:\:t\in\N_{0},\nonumber\\
        y(t) &= h_{\sigma(t)}(x(t)),
    \end{align}
   generated by
   \begin{itemize}[label=$\circ$,leftmargin=*]
    \item a family of systems
    \begin{align}
    \label{e:family}
        x(t+1) &= f_{i}(x(t),v(t)),\:\:x(0) = x_{0},\:\:p\in\P,\:\:t\in\N_{0},\nonumber\\
        y(t) &= h_{i}(x(t)),
    \end{align}
    where $x(t)\in\R^{d}$, $v(t)\in\R^{m}$ and $y(t)\in\R^{p}$ are the vectors of states, inputs and outputs at time $t$, respectively, $\P = \{1,2,\ldots,N\}$ is an index set, and
    \item a switching signal: $\sigma:\N_{0}\to\P$.
   \end{itemize}
   We assume that for each $i\in\P$, $\ker f_{i}(\cdot,0) = \{0\}$. Let $(x(t))_{t\in\N_{0}}$ be the solution to the switched system \eqref{e:swsys} corresponding to a switching signal $\sigma$ and starting at $x_{0}$ at $t=0$. The dependence of $(x(t))_{t\in\N_{0}}$ on $\sigma$ is suppressed for notational simplicity.

    Let \(\P_{S}\) and \(\P_{U}\) denote the sets of indices of input/output-to-state stable (IOSS) and unstable subsystems, respectively, \(\P = \P_S\sqcup\P_U\). We let \(E(\P)\) denote the set of ordered pairs \((i,j)\) such that a switch from subsystem \(i\) to subsystem \(j\) is admissible, \(i,j\in\P\), \(i\neq j\), and \(\Delta_{m}\) and \(\Delta_{M}\) denote the admissible minimum and maximum dwell times on each subsystem, respectively.

    \begin{remark}
    \label{rem:restrictions}
        A distinction between admissible and inadmissible transitions captures situations where switches between certain subsystems may be prohibited. For instance, in automobile gear switching, certain switching orders (e.g., from first gear to second gear, etc.) are followed \cite[\S III]{Antsaklis_survey}. A restriction on minimum dwell time on subsystems arises in situations where actuator saturations may prevent switching frequency beyond a certain limit. Also, in order to switch from one component to another, a system may undergo certain operations of non-negligible durations \cite[\S I]{Heydari_ACC16} resulting in a minimum dwell time restriction on subsystems. Restricted maximum dwell time is natural to systems whose components need regular maintenance or replacements, e.g., aircraft carriers, MEMS systems, etc. Moreover, systems dependent on diurnal or seasonal changes, e.g., components of an electricity grid, have inherent restrictions on admissible dwell times.
    \end{remark}

    Let \(0=:\tau_{0}<\tau_{1}<\cdots\) be the \emph{switching instants}; these are the instants in time when \(\sigma\) changes its values. We call a switching signal \(\sigma\) \emph{admissible} if it satisfies: \((\sigma(\tau_{k}),\sigma(\tau_{k+1}))\in E(\P)\) and \(\tau_{k+1}-\tau_{k}\in[\Delta_{m}:\Delta_{M}]\), \(k=0,1,2,\ldots\). Let \(\Sw\) denote the set of all admissible switching signals. Recall that
    \begin{defn}
   \label{d:dt-ioss}
    The switched system \eqref{e:swsys} is input/output-to-state stable (IOSS) for a $\sigma\in\Sw$ if there exist functions $\chi_{1}\in\KL$ and $\chi_{2}$, $\chi_{3}\in\K$ such that for all bounded inputs $v:\N_{0}\to\R^{m}$ and $x(0)\in\R^{d}$, we have
    \begin{align}
    \label{e:dt-ioss}
        \norm{x(t)}\leq\chi_{1}(\norm{x(0)},t)+\chi_{2}(\norm{v}_{t})+\chi_{3}(\norm{y}_{t})
    \end{align}
    for all $t\in\N_{0}$, where $\norm{\cdot}$ is the Euclidean norm and $\norm{w}_{t} := \sup\{\norm{w(s)}|s\in\N_{0},s\leq t\}$ denotes the supremum norm of a signal $w$ taking values in some Euclidean space.
   \end{defn}
   Note that if $\chi_{3}\equiv0$, then \eqref{e:dt-ioss} reduces to input-to-state stability (ISS) of \eqref{e:swsys}, and if also $v\equiv 0$, then \eqref{e:dt-ioss} reduces to global asymptotic stability (GAS) of \eqref{e:swsys}.

    We are interested in the following problem:
    \begin{prob}
    \label{prob:mainprob}
        Design (algorithmically) a switching signal \(\sigma\in\Sw\) that preserves IOSS of the switched system \eqref{e:swsys}.
    \end{prob}

    \begin{remark}
    \label{rem:non-trivial}
        Observe that even though IOSS subsystems are present in the family \eqref{e:family}, Problem \ref{prob:mainprob} does not admit a trivial solution. Indeed, a restriction on the admissible maximum dwell times on the subsystems prevents the possibility that \(\sigma\equiv i\) for a fixed \(i\in\P_{S}\).
    \end{remark}
\subsection{Prior works}
\label{s:prior_works}
    IOSS of switched systems, both in the continuous and discrete-time regime, has attracted considerable research attention over the past decade. In \cite{Mancilla2008} IOSS of switched differential inclusions was studied. Conditions on the family of systems were identified such that a switched system generated under any switching signal is stable. A class of switching signals guaranteeing IOSS of switched nonlinear systems was characterized in \cite{Liberzon_IOSS}. The stabilizing switching signals obey average dwell time property \cite[\S 3.2.2]{Liberzon} and constrained point-wise activation of non-IOSS subsystems.

    Recently in \cite{xyz3,xyz4} we addressed IOSS of switched systems under restrictions on admissible switches between the subsystems and admissible dwell times on the subsystems by employing multiple Lyapunov-like functions \cite{Branicky98} and graph-theoretic tools \cite{Bollobas}. A weighted directed graph is associated to a switched system: its vertices correspond to the subsystems, edges correspond to the admissible switches between the subsystems, and vertex and edge weights are scalars computed from Lyapunov-like functions corresponding to the subsystems. The proposed algorithm generates a stabilizing periodic switching signal in two stages: first, it identifies a cycle on the underlying weighted directed graph of a switched system that satisfies certain properties involving the vertex and edge weights and the admissible dwell times on the subsystems (we call it a \(\Delta\)-contractive cycle), and second, it constructs a switching signal by employing this cycle. We also recast the problem of detecting/designing a \(\Delta\)-contractive cycle as a ``negative cycle'' detection/design problem, for which many off-the-shelf algorithms are available. These algorithms typically involve storing connectivity and weights associated to the vertices and edges of the graph under consideration, in memory prior to their application. In addition, detection of a negative cycle often involves a search over all paths on the graph that satisfy certain properties. Hence, the design of stabilizing switching signals by involving negative cycle detection method is not suitable for large-scale switched systems. Indeed, the total number of cycles on a graph increases exponentially with the size of the graph and storing a large set of vertex and edge weights requires a huge memory. These features motivate a search for cycle detection algorithms that work well also for large graphs.
\subsection{Our contributions}
\label{ss:contri}
    In this paper we report an extension of our earlier result for GAS of switched linear systems, proposed in \cite{def}, to the setting of switched nonlinear systems with inputs and outputs under restrictions on admissible switches between the subsystems and admissible dwell times on the subsystems. We solve Problem \ref{prob:mainprob} in three steps:
    \begin{itemize}[label = \(\circ\), leftmargin = *]
        \item first, we discuss a randomized algorithm that detects cycles on the underlying weighted directed graph of a switched system without requiring complete knowledge of the vertex and edge weights and without exploring the connectivity of the entire graph.
        \item second, we identify conditions on the connectivity, vertex weights scaled by an appropriate factor and edge weights of the underlying weighted directed graph of a switched system under which a cycle obtained from our algorithm is \(\Delta\)-contractive with overwhelming probability. The condition on connectivity of the graph corresponds to admissible switches between the subsystems and the conditions on scaled vertex weights and edge weights correspond to properties of the individual subsystems (at an abstract level of certain scalars obtained from Lyapunov-like functions corresponding to the subsystems) and the admissible minimum and maximum dwell times on the subsystems.
        \item third, we employ \(\Delta\)-contractive cycles to devise an algorithm that constructs a switching signal which obeys the pre-specified restrictions on admissible switches and admissible dwell times. A switching signal obtained from our algorithm preserves IOSS of a switched system with high probability.
    \end{itemize}
    The main features of our results are the following:
    \begin{itemize}[label = \(\circ\), leftmargin = *]
        \item Our algorithm for the design of \(\Delta\)-contractive cycles does not require complete knowledge of the vertex and edge weights of the underlying weighted directed graph of a switched system for its execution, and hence also fit for large-scale switched systems (e.g., variable structure systems with a large number of substructures) for which not all the weights can be kept in memory at once. We provide probabilistic guarantees for the \(\Delta\)-contractivity of a cycle under mild statistical hypothesis on the connectivity and weights of the graph under consideration.
        \item The above algorithm exhibits an ``online learning'' character in the sense that starting with a rough probabilistic description of the underlying weighted directed graph of a switched system (i.e., without the knowledge of precise values of the weights), we explore the directed graph and synthesize a cycle during this exploration that satisfies certain conditions with high probability.
        \item If the constituent subsystems of a switched system are prone to evolve/drift over time in a manner that is not precisely known but certain statistical estimates of the nature of evolution are available, then our algorithm can be applied and it will construct a \(\Delta\)-contractive cycle with uniform probabilistic guarantees over all such evolutions.
    \end{itemize}
    The set of results presented in this paper recovers a subset of our results in \cite{def} under sufficient conditions on the properties of the individual subsystems and the admissible dwell times on the subsystems.
\subsection{Paper organization}
\label{ss:paper_org}
    The remainder of this paper is organized as follows: in \S\ref{s:prelims} we catalog a set of preliminaries for our results. Our main results appear in \S\ref{s:mainres}. We conclude in \S\ref{s:concln} with a brief discussion of future research direction.

{\bf Notation}: For a finite set \(A\), its size is denoted by \(\abs{A}\). For \(y\in\R\), \(\lfloor y\rfloor\) denotes the largest integer smaller than or equal to \(y\).
\section{Preliminaries}
\label{s:prelims}
\subsection{The family of systems}
\label{ss:family}
    \begin{assump}
    \label{a:Lyap-prop1}
        For each $i\in\P$, there exist continuous functions $V_{i}:\R^{d}\to[0,+\infty[$, class $\Kinfty$ functions $\underline{\alpha}$, $\overline{\alpha}$, class $\K$ function $\gamma_{1}$, $\gamma_{2}$ and scalars $\lambda_{i}$ with $0<\lambda_{i}<1$ for $i\in\P_{S}$ and $\lambda_{i} > 1$ for $i\in\P_{U}$ such that for all $\xi\in\R^{d}$ and $\eta\in\R^{m}$, we have
        \begin{align}
        \label{e:Lyap-prop1}
            \underline{\alpha}(\norm{\xi})\leq V_{i}(\xi) \leq \overline{\alpha}(\norm{\xi}),
        \end{align}
        and
        \begin{align}
        \label{e:Lyap-prop2}
            V_{i}(f_{i}(\xi,\eta))\leq\lambda_{i}V_{i}(\xi) + \gamma_{1}(\norm{\eta}) + \gamma_{2}(\norm{h_{i}(\xi)}),\:\:t\in\N_{0}.
        \end{align}
    \end{assump}
    The functions $(V_{i})_{i\in\P}$ satisfying conditions \eqref{e:Lyap-prop1} and \eqref{e:Lyap-prop2} are called the (IOSS-)Lyapunov-like functions and are standard in the literature, see e.g., \cite{Cai_Lyap} for details regarding existence of such functions and their properties.
    \begin{assump}
    \label{a:Lyap-prop2}
        For each pair $(i,j)\in E(\P)$, there exist $\mu_{ij}\geq 1$ such that the following inequality holds:
        \begin{align}
        \label{e:Lyap-prop3}
            V_{j}(\xi) \leq \mu_{ij}V_{i}(\xi)\:\:\text{for all}\:\xi\in\R^{d}.
        \end{align}
    \end{assump}
\subsection{The underlying weighted directed graph of \eqref{e:swsys}}
\label{ss:digraph}
    We associate a weighted directed graph $G(\P,E(\P))$ to the switched system \eqref{e:swsys} as follows:
    \begin{itemize}[label = $\circ$, leftmargin = *]
        \item The set of vertices is the index set $\P$.
        \item The set of edges $E(\P)$ consists of a directed edge from vertex $i$ to vertex $j$ whenever a transition from subsystem $i$ to subsystem $j$ is admissible.
        \item The vertex weights are: $w(j) = -\abs{\ln\lambda_{j}}$, $j\in\P_S$, $w(j) = \abs{\ln\lambda_{j}}$, $j\in\P_U$, and the edge weights are: $w(i,j) = \ln\mu_{ij}$, $(i,j)\in E(\P)$, where $\lambda_{j}$ and $\mu_{ij}$ are as described in Assumptions \ref{a:Lyap-prop1} and \ref{a:Lyap-prop2}, respectively.
    \end{itemize}

    A walk \(W\) on \(G(\P,E(\P))\) is an alternating sequence of vertices and edges \(W = v_{0},e_{1},v_{1},e_{2},v_{2},\ldots,v_{n-1},e_{n},v_{n}\), where \(v_{k}\in\P\), \(e_{k} = (v_{k-1},v_{k})\in E(\P)\), \(0<k\leq n\). The length of a walk is its number of edges, counting repetitions, e.g., the length of \(W\) is \(n\). The initial vertex of \(W\) is \(v_0\) and the final vertex of \(W\) is \(v_n\). If \(v_0 = v_n\), then we say that the walk \(W\) is closed. A closed walk \(W\) is called a cycle if the vertices \(v_k\), \(0 < k \leq n-1\) are distinct from each other and \(v_0\).

    \begin{defn}{\cite[Definition 2]{xyz4}}
    \label{d:delta-contrac}
        A cycle \cite[p.\ 5]{Bollobas} $W = v_{0}, (v_{0},v_{1}),v_{1},\ldots,v_{n-1},(v_{n-1},v_{0})$,$v_{0}$ on $G(\P,E(\P))$ is called $\Delta$-contractive if there exist $\Delta_{v_{k}}\in[\Delta_{m}:\Delta_{M}]$, $k=0,1,\ldots,n$, such that $W$ satisfies
        \begin{align}
        \label{e:delta-contrac-ineq}
            \Gamma(W) := \sum_{k=0}^{n-1}w(v_{k})\Delta_{v_{k}} + \sum_{\substack{k=0\\v_{n}:=v_{0}}}^{n-1}w(v_{k},v_{k+1}) < 0,
        \end{align}
        where $n$ is the length of $W$, $\Delta_{m}$ and $\Delta_{M}$ are the admissible minimum and maximum dwell times on each subsystem (vertex) $v_{k}\in\P$, respectively, $w(v_{k})$ is the weight of the vertex $v_{k}\in\P$ and $w(v_{k},v_{k+1})$ is the weight of the edge $(v_{k},v_{k+1})\in E(\P)$. The scalar $\Delta_{v_{k}}$ is called $\Delta$-parameter of vertex $v_{k}$, $k=0,1,\ldots,n-1$, and $\displaystyle{\Delta_{W} := \sum_{k=0}^{n-1}\Delta_{v_{k}}}$.
    \end{defn}

    We shall employ \(\Delta\)-contractive cycles on \(G(\P,E(\P))\) to construct elements of \(\Sw\) that preserve IOSS of the switched system \eqref{e:swsys}. We are now in a position to present our results.
\section{Results and discussions}
\label{s:mainres}
    We begin with a set of properties of the underlying weighted directed graph of a switched system. These properties will be useful in guaranteeing \(\Delta\)-contractivity of cycles that we shall design momentarily.

    Let \(N_{\P_{S}}^{+}(v):= \{u\in\P_{S}\:|\:(v,u)\in E(\P)\}\) denote the set of outneighbours of a vertex \(v\in\P_{S}\) and \(d_{\P_{S}}^{+}(v) := \abs{N_{\P_{S}}^{+}(v)}\) denote the outdegree of \(v\in \P_{S}\).
    \begin{defn}
    \label{d:nice_connect}
        Let \(\Phi:\N\to\R\) be a monotone increasing function. A weighted directed graph \(G(\P,E(\P))\) is said to be \emph{nicely connected} if \(d^{+}_{\P_{S}}(j)\geq\lfloor\Phi(\abs{\P_{S}})\rfloor\) for all \(j\in\P\).
    \end{defn}

    The notion of nice connectivity of \(G(\P,E(\P))\) describes the richness of admissible switches among the subsystems and is borrowed verbatim from \cite[Definition 4.1]{def}. The function \(\Phi\) serves the purpose of quantifying the density of edges in \(G(\P,E(\P))\) in terms of its order. Every vertex in a nicely connected directed graph has at least \(\lfloor\Phi(\abs{\P_{S}})\rfloor\)-many IOSS outneighbours.

    \begin{defn}
    \label{d:nice_weight}
        A weighted directed graph \(G(\P,E(\P))\) is said to be \emph{nicely \(\Delta\)-weighted} if the following hold:
        \begin{itemize}[label = \(\circ\), leftmargin = *]
            \item there exist \(\Delta\in[\Delta_{m}:\Delta_{M}]\), \(\beta\), \(B > 0\) satisfying \(0 < \beta < B\) such that the vertex weights \(w(j)\) satisfy \(0 < w(j)\Delta\leq B\) and \(\EE[w(j)\Delta\:|\:\{w_{\ell}\Delta\}_{\ell\neq j},\{w(k,\ell)\}_{(k,\ell)\in E(\P)}] = \beta\) for all \(j\in\P\), and
            \item there exist constants \(A > 0\) and \(\alpha < \beta\) such that for every \((i,j)\in E(\P)\), the edge weight \(w(i,j)\in[-A,A]\) and \(\EE[w(i,j)\:|\:\{w_{i}\Delta\}_{i\in\P},\{w(k,\ell)\}_{(k,\ell)\neq(i,j)}]\leq\alpha\).
        \end{itemize}
    \end{defn}

    The notion of nice \(\Delta\)-weights of \(G(\P,E(\P))\) is a slightly modified version of what is called nice weights of \(G(\P,E(\P))\) in \cite[Definition 4.1]{def}. We accommodate the admissible dwell times on the subsystems into the latter to define the former. In particular,
    \begin{proposition}
    \label{prop:recover1}
        Let \(\Delta_{m} = 1\) and the conditions of Definition \ref{d:nice_weight} hold with \(\Delta = \Delta_{m} = 1\). Then \(G(\P,E(\P))\) is nicely weighted.
    \end{proposition}

    The nice \(\Delta\)-weighted property provides quantitative information about the (in)stability of the subsystems combined with the admissible dwell times on them at an abstract level. The scalars \(w(j)\), \(j\in\P\) and \(w(i,j)\), \((i,j)\in E(\P)\) are obtained from Lyapunov-like functions corresponding to the individual subsystems (\'{a} la Assumptions \ref{a:Lyap-prop1} and \ref{a:Lyap-prop2}) and the choice of these functions is not unique. The condition that the vertex weights scaled by a choice of admissible dwell times, \(w(j)\Delta\), \(j\in\P\) and the edge weights, \(w(i,j)\), \((i,j)\in E(\P)\) are uniformly bounded is no loss of generality on account of the graph \(G(\P,E(\P))\) and the interval \([\Delta_{m}:\Delta_{M}]\) being finite. In addition, Definition \ref{d:nice_weight} does not require independence of the vertex weights scaled by an admissible dwell time and the edge weights. This feature is important because, for instance, given \(w(i,j)\) (i.e., \(\ln\mu_{ij}\)), the scalar \(w(j,i)\) (i.e., \(\ln\mu_{ji}\)) cannot be independent of \(w(i,j)\). Notice that we do not assume any particular probabilistic model (that may be tuned to specific applications) for the weights and the admissible dwell times.

    We now employ a probabilistic algorithm \cite[Algorithm 1]{def} for the detection of cycles on \(G(\P,E(\P))\). If \(G(\P,E(\P))\) is nicely connected and nicely \(\Delta\)-weighted, then a cycle obtained from our algorithm is \(\Delta\)-contractive with high probability.

    \begin{algorithm}[htbp]
			\caption{Cycle detection algorithm} \label{algo:cycle_detection}
		\begin{algorithmic}[1]
			\STATE Set \(k=0\).
            \STATE Pick \(v_{k}\in\P_{S}\) uniformly at random.
            \IF {\(N_{\P_{S}}^{+}(v_{k})\setminus\{v_{0},\ldots,v_{k}\}\neq\emptyset\)}\label{step:repeat_step1}
                \STATE Pick \(v_{k+1}\in N_{\P_{S}}^{+}(v_{k})\setminus\{v_{0},\ldots,v_{k}\}\) uniformly at random.
                \STATE Set \(k=k+1\).
                \STATE Go to Step \ref{step:repeat_step1}.
                \ELSE
                    \STATE Pick \(v_{k+1}=v_{i}\) such that \(v_{i}\in N_{\P_{S}}^{+}(v_{k})\) and \((k-i)\) is maximum.
                    \STATE Exit Algorithm \ref{algo:cycle_detection}.
            \ENDIF			
		\end{algorithmic}
	\end{algorithm}

    On the digraph \(G(\P,E(\P))\), Algorithm \ref{algo:cycle_detection} generates a walk in the following fashion: At the first step, a vertex corresponding to an IOSS subsystem is picked uniformly at random. At every \(k+1\)-th step we identify the subset of outneighbours of the vertex picked at the \(k\)-th step such that the vertices correspond to IOSS subsystems, and they have not been picked till (and including) step \(k\); then a vertex from the above subset is picked uniformly at random. If there is no such outneighbour corresponding to the IOSS vertices that were not picked earlier, the algorithm selects an outneighbour that is at the maximum `distance' from the current vertex in the already generated sequence, and the algorithm is stopped. Since we deal with finite directed graphs, it is evident that every walk generated by our algorithm is closed. In addition, the mechanism of repeating vertices (as described above) makes this closed walk a cycle. It follows that the length of the cycles is bounded above by the order of the graph, and hence the time complexity associated to Algorithm \ref{algo:cycle_detection} is linear in the order of the graph.

    \begin{proposition}{\cite[Proposition 1]{def}}
    \label{prop:mainres1}
        If the weighted directed graph \(G(\P,E(\P))\) is nicely connected, then a cycle \(W\) obtained from Algorithm \ref{algo:cycle_detection} has the following properties:
        \begin{itemize}[label = \(\circ\), leftmargin = *]
            \item all vertices in \(W\) are from \(\P_{S}\), and
            \item the length of \(W\) is at least \(\lfloor\Phi(\abs{\P_{S}})\rfloor\).
        \end{itemize}
    \end{proposition}

    \begin{proposition}
    \label{prop:mainres2}
        Consider the family of systems \eqref{e:family}, the set of admissible switches between the subsystems, \(E(\P)\), and the admissible minimum and maximum dwell times on the subsystems, \(\Delta_{m}\) and \(\Delta_{M}\), respectively. Suppose that the underlying weighted directed graph \(G(\P,E(\P))\) of the switched system \eqref{e:swsys} is nicely connected and nicely \(\Delta\)-weighted. Then a cycle of length at least \(\lfloor\Phi(\abs{\P_{S}})\rfloor\) on \(G(\P,E(\P))\) obtained from Algorithm \ref{algo:cycle_detection} is \(\Delta\)-contractive with probability at least \(1-\exp\Biggl(-\frac{1}{2}\biggl(\frac{\alpha-\beta}{A+B}\biggr)^{2}\lfloor\Phi(\abs{\P_{S}})\rfloor\Biggr)\).
    \end{proposition}

    \begin{proof}
        Let \(W = v_{0},(v_{0},v_{1}),v_{1},\ldots,v_{n-1},(v_{n-1},v_{0}),v_{0}\) be a cycle obtained from Algorithm \ref{algo:cycle_detection}. It follows under the set of arguments employed in \cite[Proof of Theorem 4.2]{def} with \(w(v_{k})\) replaced by \(w(v_{k})\Delta\) that \(W\) is \(\Delta\)-contractive, where \(\Delta_{v_{k}} = \Delta\), \(k=0,1,\ldots,n-1\), with probability at least \(1-\exp\Biggl(-\frac{1}{2}\biggl(\frac{\alpha-\beta}{A+B}\biggr)^{2}\lfloor\Phi(\abs{\P_{S}})\rfloor\Biggr)\).
    \end{proof}

    \begin{remark}
    \label{rem:prob-vs-det}
        For \(N\) not large, recently in \cite{xyz4}, we applied deterministic negative cycle detection algorithms on \(G(\P,E(\P))\) for the synthesis of \(\Delta\)-contractive cycles. These algorithms require a priori knowledge of the vertex and edge weights and often explores the entire graph. As a result, they are not suitable for large-scale switched systems. Notice that Algorithm \ref{algo:cycle_detection} does not require the scalars \(w(j)\Delta\), \(j\in\P\), \(w(i,j)\), \((i,j)\in E(\P)\) to be stored in memory prior to its application and does not conduct searches over multiple paths on the graph that satisfy certain properties. As a result, it is suitable for large weighted directed graphs corresponding to large-scale switched systems. For such graphs, Algorithm \ref{algo:cycle_detection} provides probabilistic guarantees for the design of \(\Delta\)-contractive cycles. In addition, this algorithm has an ``online learning'' property in the sense that starting with a rough probabilistic description of \(G(\P,E(\P))\) (i.e., without the knowledge of the precise values of the weights), we explore the graph and synthesize a cycle during this exploration that is \(\Delta\)-contractive with high probability.
    \end{remark}

    We now employ \cite[Algorithm 1]{xyz4} to construct a periodic switching signal \(\sigma\) by employing a cycle \(W = v_{0},(v_{0},v_{1}),v_{1},\ldots,v_{n-1},(v_{n-1},v_{0}),v_{0}\) on \(G(\P,E(\P))\) obtained from Algorithm \ref{algo:cycle_detection}.

    \begin{algorithm}[htbp]
			\caption{Construction of \(\sigma\)} \label{algo:swsig_design}
		\begin{algorithmic}[1]
			\STATE Set \(k=0\) and \(\tau_{0} = 0\).
            \FOR {\(p=kn,kn+1,\ldots,(k+1)n-1\)}\label{step:repeat_step2}
                \STATE Set \(\sigma(\tau_{p}) = v_{p-kn}\) and \(\tau_{p+1} = \tau_{p} + \Delta\).
            \ENDFOR
            \STATE Set \(k=k+1\) and go to Step \ref{step:repeat_step2}.			
		\end{algorithmic}
	\end{algorithm}
    It is immediate that a switching signal obtained from Algorithm \ref{algo:swsig_design} is periodic with period $\displaystyle{\Delta_{W} := \sum_{k=0}^{n-1}\Delta = n\Delta}$.
    \begin{proposition}
    \label{prop:mainres3}
        Consider the family of systems \eqref{e:family}, the set of admissible switches between the subsystems, \(E(\P)\), and the admissible minimum and maximum dwell times on the subsystems, \(\Delta_{m}\) and \(\Delta_{M}\), respectively. Suppose that the underlying weighted directed graph \(G(\P,E(\P))\) of the switched system \eqref{e:swsys} is nicely connected and nicely \(\Delta\)-weighted. Then a switching signal \(\sigma\) obtained from Algorithm \ref{algo:swsig_design} satisfies the following properties:
        \begin{enumerate}[label = \roman*), leftmargin = *]
            \item \(\sigma\in\Sw\), and
            \item \(\sigma\) preserves IOSS of the switched system \eqref{e:swsys} with probability at least \(1-\exp\Biggl(-\frac{1}{2}\biggl(\frac{\alpha-\beta}{A+B}\biggr)^{2}\lfloor\Phi(\abs{\P_{S}})\rfloor\Biggr)\).
        \end{enumerate}
    \end{proposition}

    \begin{proof}
        i) is immediate.

        From \cite[Theorem 1]{xyz4} it follows that a switching signal \(\sigma\) constructed in Algorithm \ref{algo:swsig_design} by employing a \(\Delta\)-contractive cycle satisfies ii). Now, the cycles constructed via Algorithm \ref{algo:cycle_detection} are \(\Delta\)-contractive with probability at least \(1-\exp\Biggl(-\frac{1}{2}\biggl(\frac{\alpha-\beta}{A+B}\biggr)^{2}\lfloor\Phi(\abs{\P_{S}})\rfloor\Biggr)\). Employing such a cycle in Algorithm \ref{algo:swsig_design} leads to the assertion of Proposition \ref{prop:mainres3}.
    \end{proof}

    The following claim is straightforward and recovers the first assertion of \cite[Theorem 4.3]{def} on GAS of a switched linear system under unrestricted dwell times on the subsystems:
    \begin{proposition}
    \label{prop:recover2}
        Consider the family of systems \eqref{e:family}. Suppose that \(f_{i}(x,v) = A_{i}x\), \(A_{i}\in\R^{d\times d}\), \(h_{i}\equiv 0\) for all \(i\in\P\) and \(\Delta_{m} = 1\). In addition, the underlying weighted directed graph \(G(\P,E(\P))\) of the switched system \eqref{e:swsys} is nicely connected and nicely \(\Delta\)-weighted with \(\Delta = 1\). Then a switching signal \(\sigma\) obtained from Algorithm \ref{algo:swsig_design} preserves GAS of the switched system \eqref{e:swsys} with probability at least \(1-\exp\Biggl(-\frac{1}{2}\biggl(\frac{\alpha-\beta}{A+B}\biggr)^{2}\lfloor\Phi(\abs{\P_{S}})\rfloor\Biggr)\).
    \end{proposition}

    To summarize, we discussed an algorithm to design periodic switching signals that obey pre-specified restrictions on admissible switches between the subsystems and admissible minimum and maximum dwell times on the subsystems. We also identified sufficient conditions on the following components under which a switching signal obtained from our algorithm preserves IOSS of the switched system \eqref{e:swsys} with high probability:
    \begin{enumerate}[label = \roman*), leftmargin = *]
        \item the (in)stability of the elements in the family \eqref{e:family} (at an abstract level of certain scalars obtained from Lyapunov-like functions corresponding to the subsystems) in terms of scaled vertex weights and edge weights of the underlying weighted directed graph of a switched system,
        \item the admissible switches between the subsystems in terms of connectivity of underlying weighted directed graph of a switched system, and
        \item the admissible minimum and maximum dwell times on the subsystems in terms of the scaling factors of vertex weights of the underlying weighted directed graph of a switched system.
    \end{enumerate}

    We now present a numerical experiment to demonstrate our results.
\section{A numerical experiment}
\label{s:num_ex}
    We consider a nicely connected and nicely \(\Delta\)-weighted directed graph \(G(\P,E(\P))\) with
    \begin{itemize}[label = \(\circ\), leftmargin = *]
        \item \(\abs{\P_S} = 1000\), \(\Phi(r) = \frac{1}{10}\sqrt{r}\), \(d^{+}(j) = \lfloor\Phi(\abs{\P_S})\rfloor\) for all \(j\in\P\),
        \item \(\Delta_m = 2\), \(\Delta_M = 4\), and
        \item \(A = 2.5\), \(B = 5\), \(\alpha = 0\) and \(\beta = 2.5\).
    \end{itemize}
    We extract and fix a cycle \(W\) obtained from Algorithm \ref{algo:cycle_detection} on \(\P_S\subset\P\). The vertex and edge weights of \(W\) and the dwell times on the vertices (subsystems) are sampled uniformly at random \(1000\) times from the intervals as stipulated in Definition \ref{d:nice_weight}. We calculate
    \begin{align*}
        X_n = \sum_{k=1}^{n}w(v_{k-1},v_{k})-\sum_{k=0}^{n}w(v_k)\Delta
    \end{align*}
    empirically for \(n\) being the length of the cycle \(W\).

    The above experiment is repeated for cycles of different length \(n\) obtained from Algorithm \ref{algo:cycle_detection} with uniformly randomly selected initial vertex. We plot the empirical probability of \(\{X_n < 0\}\) versus \(n\) in Figure \ref{fig:plot}. An overwhelming empirical probability is observed.

    \begin{figure}[htbp]
    \centering
        \includegraphics[scale=0.7]{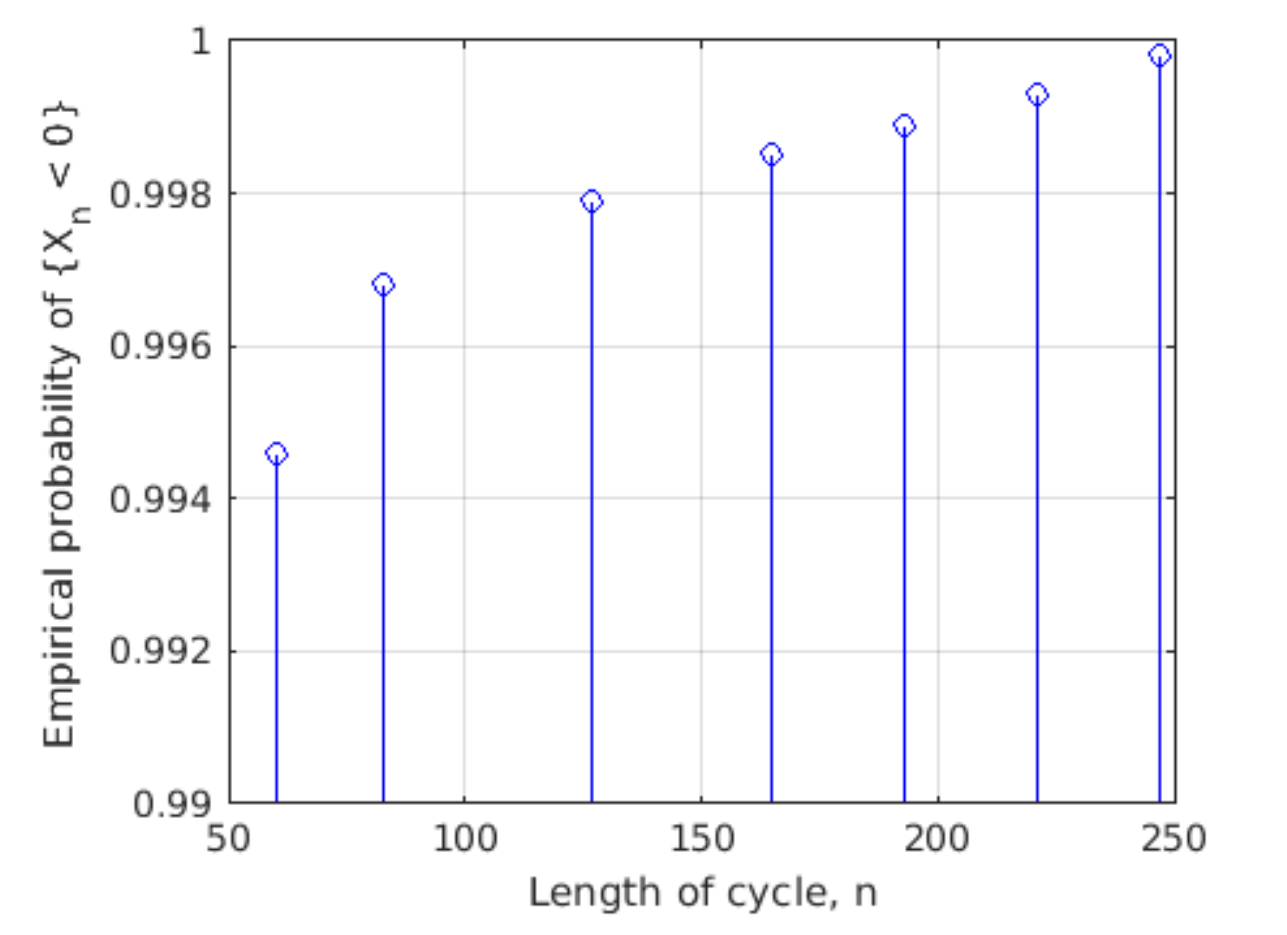}
        \caption{Plot for the empirical probability of a cycle being \(\Delta\)-contractive against its length}\label{fig:plot}
    \end{figure}
\section{Conclusion}
\label{s:concln}
    Notice that our stabilizing switching signals activate only IOSS subsystems. In addition, a successful application of our cycle detection algorithm assumes certain knowledge of the Lyapunov-like functions corresponding to the individual subsystems in an expected sense. Future research directions include accommodating unstable subsystems (vertices) in a \(\Delta\)-contractive cycle and exploring the possible extensions of the ``online learning'' nature associated to Algorithm \ref{algo:cycle_detection} to the learning of quantitative measures of (in)stability of the subsystems. These matters are currently under investigation and will be reported elsewhere.


\end{document}